\documentclass[article]{amsart}

\usepackage{amsmath}
\usepackage{amsfonts}
\usepackage{amsmath,mathtools}
\usepackage{amssymb}
\usepackage{url}
\usepackage{amscd}
\usepackage{bbm}
\usepackage{amsthm}
\usepackage{tikz}
\usepackage{graphicx}
\usetikzlibrary{matrix}
\usepackage{caption}
\usepackage{mathrsfs}
\usepackage [english]{babel}
\usepackage [autostyle, english = american]{csquotes}
\MakeOuterQuote{"}
\numberwithin{equation}{section}
\def\bysame{\leavevmode\hbox to3em{\hrulefill}\thinspace}

\theoremstyle{plain}

\newtheorem*{theorem*}{Theorem}

\newtheorem{lemma}{Lemma}

\newtheorem*{conjecture*}{Conjecture}
\newtheorem*{lemma*}{Lemma}

\theoremstyle{definition}
\newtheorem{definition}{Definition}
\newtheorem*{definition*}{Definition}

\theoremstyle{remark}

\begin{document}

\date{} 
\title[On the Shrinkable U.S.C. Decomposition Spaces of Spheres]{On the Shrinkable U.S.C. Decomposition Spaces of Spheres}
\author{Shijie Gu}

\address{Department of Mathematics \& Statistics\\
University of Nevada, Reno\\
1664 N. Virginia Street
Reno, NV, USA}

\email{sgu@unr.edu}


\begin{abstract}
Let $G$ be a u.s.c decomposition of $S^n$, $H_G$ denote the set of nondegenerate elements and $\pi$ be the projection of $S^n$ onto $S^n/G$. Suppose that each point in the decomposition space has arbitrarily small neighborhoods with ($n-1$)-sphere frontiers which miss $\pi(H_G)$, and such frontiers satisfies the Mismatch Property. Then this paper shows that this condition implies $S^n/G$ is homeomorphic to $S^n$ ($n\geq 4$). This answers a weakened form of a conjecture asked by Daverman [3, p. 61]. In the case $n=3$, the strong form of the conjecture has an affirmative answer from Woodruff [12].
\end{abstract}
\maketitle


\newcommand\sfrac[2]{{#1/#2}}

\newcommand\cont{\operatorname{cont}}
\newcommand\diff{\operatorname{diff}}

\section{Introduction}
The following open question was asked by Daverman [3, p. 61] about thirty years ago,
\begin{conjecture*}
Suppose $G$ is a u.s.c. decomposition of $E^n$ such that for each $g\in H_G$ and each open neighborhood $W_g$ of $g$ there exists a neighborhood $U_g$ with $g\subset U_g\subset W_g$, where the frontier of $U_g$ is an $(n-1)$-sphere missing $N_G$. G is shrinkable.
\end{conjecture*}
What makes the proof of this conjecture so challenging is the existence of $(n-1)$-spheres in $S^n$ in which all disks are wild [5]. That means a failure will be inevitable if one uses a typical shrinking approach---dividing the spheres into cellular subdivisons. Here we treat an important special case of the conjecture, which one can argue is a genuine analog of the three dimensional result depicted in the next paragraph. That is, the conjecture will be true if the Mismatch Property is satisfied. It turns out to be our main theorem in section 3. In this section, readers may find out that even though in such a weak form, the proof is far from straightfoward. 

The strategy of the proof is greatly inspired by Woodruff [12] and Eaton [8]. However, some of their methods only apply in the 3-dimensional case. Therefore, for $n \geq 4$, their techniques have to be modified and extended. In section 5, Eaton's shrinking process is modified into a more manageable way which is an application of Daverman's proof of Mismatch Theorem [2, 6]. Then Woodruff's Lemma 1 and 2 [12] can be extended. That means the generalization of her Theorem 2 [12] will be executable. So section 5 is the main part of our proof.

Some counterexamples given by Daverman [6] show that $(n-1)$-spheres in $S^n$, unlike those in $S^3$, do not necessarily satisfy the Mismatch Property. 

The author wishes to thank Robert J. Daverman for many very helpful discussions and suggesting improvements in the preliminary version.


\section{Definitions and Notations}
All decompositions used in this paper are upper semicontinuous (u.s.c) defined by Daverman [3, P. 8]. For a decomposition $G$ of $S^n (n \geq 4)$, the set of nondegenerate elements is denoted by $H_G$, and the natural projection of $S^n$ onto $S^n/G$ by $\pi$. A subset $X\subset S^n$ is $saturated$ (or \textit{G-saturated}) if $\pi^{-1}(\pi(X))=X$. 

The symbol $\rho$ is used to denote the distance fixed on $S^n$ between sets $A$ and $B$ as $\rho(A,B)$; for the closure of $A$ we use $\text{Cl }A$; the boundary, interior and exterior of $A$ are denoted as $\text{Bd }A$, $\text{Int }A$ and $\text{Ext }A$; the symbol $\mathbbm{1}$ is the identity map and, for $A \subset X$, $\mathbbm{1}|_A$ is used to denote the inclusion of $A$ in $X$. For $A\subset S^n$ and $\varepsilon>0$, an embedding $h$ of $A$ in $S^n$ is an $\varepsilon$\textit{-homeomorphism} if and only if $\rho(h, \mathbbm{1}|A)<\varepsilon$ [4].
For a collection $H$, let $H^{*}=\{x\in g:g\in H\}$.

A $crumpled$ $n$-$cube$ $C$ is a space homeomorphic to the union of an ($n-1$)-sphere in $S^n$ and one of its complementary domains; the subset of $C$ consisting of those points at which $C$ is an $n$-manifold (without boundary) is called the $interior$ $of$ $C$, written as Int $C$, and the subset $C-\text{Int }C$, which corresponds to the given ($n-1$)-sphere, is called the $boundary$ $of$ $C$, written as $\text{Bd }C$. A crumpled $n$-cube $C$ is a $closed$ $n$-$cell$-$complement$ if there exists an embedding $h$ of $C$ in $S^n$ such that $S^n-h(\text{Int }C)$ is an $n$-cell.

For $n=3$, it has been shown by Hosay [10] and Lininger [11] that each crumpled cube $C$ can be embedded in $S^3$ so that Cl$(S^3-C)$ is a 3-cell. In the case of $n\geq 4$, the same claim is proved by Daverman [4].

Let $A$ be an annulus bounded by ($n-1$)-spheres $\Sigma_1$ and $\Sigma_2$. A homeomorphism $\varphi$ taking $\Sigma_1$ onto $\Sigma_2$ is called $admissible$ if there exists a homotopy $H:$ $S^{n-1} \times I \rightarrow A$ such that $H(S^{n-1} \times 0) = \Sigma _1$; $H(S^{n-1} \times 1)=\Sigma_2$; and for $x \in S^{n-1}$ if $H(x \times 0)=p \in \Sigma_1$, then $H(x \times 1)=\varphi(p) \in \Sigma_2.$
\section{Mismatch Property and Shrinkability}
Before stating our theorem, we should give the definition of Mismatch Property for $(n-1)$-spheres in the high dimensional case $(n\geq 4)$,
\begin{definition}[Mismatch Property]
An $(n-1)$-sphere in $S^n$ has the \textit{Mismatch Property} if and only if it bounds two crumpled $n$-cubes $C_0$ and $C_1$, and there exists $F_\sigma$-sets $F_i$ in $\text{Bd } C_i$ such that $F_i \cup \text{Int }C_i$ is 1-ULC, $i=0,1$, and $F_0\cap F_1=\emptyset$.

\end{definition}
For a u.s.c. decomposition $G$ of $S^n$ and an $(n-1)$-sphere
$\Sigma \in S^n/G$ that misses $\pi(H_G)$, we say that $\Sigma$ satisfies the
Mismatch Property if $\pi^{-1}(\Sigma)$ does. Then the main theorem can be stated as below,

\begin{theorem*}
Suppose $G$ is a u.s.c. decomposition of $S^n$ such that for any $p\in \pi(H_G)$ and open set $U$ containing $p$ there is an open set $V$ such that $p\in V \subset U$ and $\text{Bd }V$ is an $(n-1)$-sphere which misses $\pi(H_G)$ and satisfies the Mismatch Property. Then $G$ is shrinkable and $S^n/G$ is homeomorphic to $S^n$.
\end{theorem*}

The existence of a homeomorphism between $S^n$ and $S^n/G$ can be obtained by generalized Bing's fundamental shrinking theorem [3, P. 23].
\section{Generalization of Woodruff's Lemma and Proof of Theorem}
In this section, the main theorem will be reduced to Lemma 1, which is similar to Woodruff's Lemma 1 ($n = 3$), but generalized to the $n$-sphere ($n\geq 4$).
The proof of Lemma 1 will be given in section 5.
\begin{lemma}
Suppose $G$ is a u.s.c. decomposition of $S^n$, $\varepsilon >0$, $C$ is a crumpled $n$-cube in $S^n$ with $\text{Bd } C$ an $(n-1)$-sphere which fails to meet any nondegenerate element of $G$, and satisfies the Mismatch Property. Then there exists a homeomorphism $h: S^n \rightarrow S^n$ such that
\begin{itemize}
\item[(1)] $h|S^n-N(C,\varepsilon)=\mathbbm{1}$,
\item[(2)] if $g\in G$ and $g\subset C$, $\mathrm{Diam}$ $h(g)<\varepsilon$, and
\item[(3)] if $g\in G$, then $\mathrm{Diam}$ $h(g)<\varepsilon + \mathrm{Diam}$ $g$.
\end{itemize}
\end{lemma}

Next, we shall show the main theorem can be proved in three ways by applying Lemma 1. 
\begin{proof}[Proof 1]
One can imitate the strategy used by Woodruff in proving her Theorem 2 [12]. Her proof will work for our situation with the simple replacement of reference to 2-spheres in $S^n$ by $(n-1)$-spheres in $S^n$ that satisfy the Mismatch Property. The reader who understands that argument should be able to fill the details required.
\end{proof}
\begin{proof}[Proof 2] This proof depends on the shrinkability criterion below in the compact metric case which is due to Edwards [9]:
\begin{definition}[Compact Metric Shrinkability Criterion]
$G$ is \textit{shrinkable} if for each $\varepsilon>0$, there exists a homeomorphism $h$ of compact metric space $S$ onto itself satisfying: (1) $\rho(\pi(s),\pi(h(s)))<\varepsilon$ for each $s\in S$, and (2) $\text{Diam } h(g)<\varepsilon$ for each $g\in G$.
\end{definition}
Let $W$ be an open set containing $\pi(H_{G})$; given a homeomorphism $\phi$ of $\pi^{-1}(W)$ onto itself; $C_0$ is a crumpled $n$-cube satisfying the hypotheses in Lemma 1, and $\varepsilon >0$ such that $\text{Cl }N(C_0,\varepsilon)$ is compact and contained in $W$. We can employ the uniform continuity of $\phi|\text{Cl }N(C_0,\varepsilon)$ to obtain a homeomorphism $\phi'=\phi h$ of $\pi^{-1}(W)$ satisfying:
\begin{itemize}
\item[(a$_0$)] $\phi'|\pi^{-1}(W)-N(C_0,\varepsilon)=\phi|\pi^{-1}(W)-N(C_0,\varepsilon)$,
\item[(b$_0$)] if $g\in G$ and $g\subset C_0$, $\text{Diam }\phi'(g)<\varepsilon$,
\item[(c$_0$)] if $g\in G$, then $\text{Diam }\phi'(g)<\varepsilon+\text{Diam }\phi(g)$.
\end{itemize}
Then we can use this modified conclusions in Lemma 1 to prove the main theorem.

For $g\in H_G$, by the fact that $W$ contains $\pi(H_G)$, and the hypothesis of the main theorem, there exists an open set $V_g$ such that $\pi(g)\in V_g \subset W$, $\text{Cl }V_g \subset W$, $\text{Bd }V_g$ is an $(n-1)$-sphere, $\text{Bd }V_g \cap \pi(H_G)=\emptyset$ and satisfies the Mismatch Property. So $\pi^{-1}(V_g)$ is saturated open set containing $g$, and $\pi^{-1}(\text{Bd }V_g)\subset \pi^{-1}(W)$ is an $(n-1)$-sphere which misses $H_G$. 

Let $\varepsilon>0$. Define a collection $H_L=\{g\in H_G: \text{Diam }g\geq \varepsilon/2\}$. Take a cover $\mathscr{U}$ of it by saturated sets of the form $\pi^{-1}(V_g)$. By the compactness of $H_{L}^*$, we are enumerating the finite cover $\{U_1,U_2,\cdots,U_k\}$ from $\mathscr{U}$. Denote the crumpled $n$-cube $\text{Cl }U_i$ as $C_i$ $(i=1,\dots,k)$. We assume that $\varepsilon$ is so small that $N(C_i,\varepsilon) \subset \pi^{-1}(W)$. Then we shall produce a shrinking homeomorphism $h_k$ at the end of finite sequence $h_0=\mathbbm{1},h_1,h_2,\dots,h_k$ of successive homeomorphisms $h_i|\pi^{-1}(W)-N(C_i,\varepsilon)=h_{i-1}|\pi^{-1}(W)-N(C_i,\varepsilon)$ $(i=1,\dots,k)$, which will satisfies the shrinkability criterion for compact metric space.

By hypothesis, there exists a homeomorphism $h_1$ of $S^n$ onto itself such that 
\begin{itemize}
\item[(a$_1$)] $h_1|\pi^{-1}(W)-N(C_1,\varepsilon)=\mathbbm{1}$,
\item[(b$_1$)] if $g\in G$ and $g\subset C_1$, $\text{Diam }h_1(g)<\varepsilon/4$,
\item[(c$_1$)] if $g\in G$, then $\text{Diam }h_1(g)<\varepsilon/4+\text{Diam }g$.
\end{itemize}

Recursively, $h_0,h_1,\dots,h_i$ and $C_1,\dots,C_i$ can be constructed subject to the 3 conditions below,
\begin{itemize}
\item[(a$_i$)] $h_i|\pi^{-1}(W)-N(C_i,\varepsilon)=h_{i-1}|\pi^{-1}(W)-N(C_i,\varepsilon)$,
\item[(b$_i$)] if $g\in G$ and $h_{i-1}(g)\subset C_i$, $\text{Diam }h_i(g)<\varepsilon/2^i$,
\item[(c$_i$)] if $g\in G$, then $\text{Diam }h_i(g)<(\varepsilon/2^i)+\text{Diam }h_{i-1}(g)$.
\end{itemize}
For each $g\in H_L$, there exists an index $i$ with $g\in C_i$, then $\text{Diam }h_i(g)<\varepsilon/2^i$ by (b$_i$) above, and by (c$_i$), for $j=1,\dots,k-i$, 
\begin{equation}
\text{Diam }h_{i+j}(g)<\sum\limits_{m=1}^{j}(\varepsilon/2^{i+m})+\text{Diam }h_i(g)<\sum\limits_{m=0}^{j}(\varepsilon/2^{i+m})<\varepsilon.
\end{equation}
Similarly, for $g\in H_G-H_L$, one can provide a similar bound. These arguments imply the final homeomorphism $h_k$ shrinks all the elements $G$ to $\varepsilon$-small size.

In addition, because $\pi h_i, \pi h_{i-1}$ agree outside of $N(C_i,\varepsilon)$ and because $\pi h_{i-1}(\pi h_i)^{-1}(N(C_i,\varepsilon))=N(C_i,\varepsilon)$, we have 
\begin{equation}
\rho(\pi h_{i-1},\pi h_i)<\varepsilon/2^i.
\end{equation}
Hence, we have $\pi h_k$ and $\pi$ are $\varepsilon$-close.
\end{proof}
\begin{proof}[Proof 3] We say $G$ is \textit{locally semi-controlled shrinkable} [7] if to every $g_0\in G$ and neighborhood $U_0$ of $g_0$ there corresponds a neighborhood $W_0 \subset U_0$ of $g_0$ such that for every $\varepsilon>0$ and homeomorphism $h:S \rightarrow S$, where $S$ is a metric space, and there exists another homeomorphism $h':S \rightarrow S$ satisfying:
\begin{itemize}
\item[(1)] $h'$ and $h$ coincide on $S-U_0$,
\item[(2)] $\text{Diam }h'(g)<\varepsilon$ for all $g\in G$ with $g\subset W_0$, and 
\item[(3)] $\text{Diam }h'(g')<\varepsilon+\text{Diam }h(g')$ for all $g'\in G$.
\end{itemize}
Lemma 1 immediately implies that the decomposition $G$ has this feature. By Daverman-Repov\v{s}'s Theorem 1.7 [7], $G$ is shrinkable.
\end{proof}
\section{Proof of Lemma 1}
As we have shown, the main theorem follows from Lemma 1. We now consider its proof.
\begin{proof}[Proof of Lemma 1]
Let $S^n$ be the union of crumpled $n$-cubes $C$ and $\tilde{C}=\text{Cl }(S^n-C)$. This can be proved by definition and Jordan-Brouwer Separation Theorem. See Figure 1.
\begin{figure}[h!]
\begin{center}
\begin{tikzpicture}
  \matrix (m) [matrix of math nodes,row sep=3em,column sep=0.5em,minimum width=2em] {
    S^n= & C& \bigcup & \tilde{C}& & & & & \\
     & S^n& & & & & & &\\
     & & & S^n& & & & &S^n\\};
  \path[-stealth]
    (m-1-2) edge node [left] {$h_C$} (m-2-2)
    (m-1-4) edge node [right] {$h_{\tilde{C}}$} (m-3-4)
    (m-3-4) edge node [above] {$\alpha$} (m-3-9)
    (m-2-2) edge node [below left, pos=.5] {$\theta$} (m-3-4);
\end{tikzpicture}
\end{center}
             \caption{}
             \label{fig:1}
\end{figure}
The map $h_C$ is a reembedding of $C$ in $S^n$ ($n\geq 4$). It's given by Daverman's Reembedding Theorem 6.1 [4], which states: \textit{Let $C$ denote a crumpled n-cube in $S^n$ $(n\geq 4)$. For each $\delta >0$, there exists an embedding $h$ of $C$ in $S^n$ such that $\rho(h,\mathbbm{1}|C) < \delta $ and $S^n-h(\text{Int } C)$ is an n-cell.}

Let $p\in \text{Int }C$ and require that $h_C$ is $\mathbbm{1}$ on a small neighborhood of $p$. Let $\theta $ be a homeomorphism of $S^n$ onto itself taking $h_C(C)$ to a set of diameter less than 
\begin{equation*}
\min\{\varepsilon/2, \rho(p,\text{Bd }C)/2\},
\end{equation*}
and not moving points in a small neighborhood of $p$. Note that the elements of $\theta (G)$ in $\theta h_{\tilde{C}}$ are $(\varepsilon/2)$-small.

Next, we need apply Daverman's Reembedding Theorem to $\tilde{C}$ to get the reembedding $h_{\tilde{C}}$ of $\tilde{C}$ in $S^n$. Choose a $\delta$ which is less than
\begin{equation*}
\min\{\varepsilon/2, \rho(\tilde{C}, \theta h_C(C))\}.
\end{equation*}
Let $h_{\tilde{C}}$ be $\mathbbm{1}$ on $S^n-N(C,\varepsilon)$. The $\varepsilon/2$ condition assures that the diameters of elements of $G$ in $\tilde{C}$ grow by no more than $\varepsilon$ under the action of $h_{\tilde{C}}$, i.e. if $q\in \tilde{C}$, $\rho(q,h_{\tilde{C}}(q))<\delta$. Meanwhile, the condition on $\delta$ successfully guarantees $\theta h_C(C) \cap h_{\tilde{C}}(\tilde{C})=\emptyset$.

The above motion controls imply that $\theta h_C(C)$ and $h_{\tilde{C}}(\tilde{C})$ are disjoint crumpled cubes in $S^n$, and the closure of the complement of each is $n$-cell. Denote Cl$(S^n-\theta h_C(C)-h_{\tilde{C}}(\tilde{C}))$ by $A$. To complete the proof of Lemma 1, Lemma 2 will be applied, which is stated below.
\end{proof}
\begin{lemma}
Suppose that in $S^n$ an annulus $A$ is bounded by $(n-1)$-spheres $\Sigma$ and $\varphi(\Sigma)$, where $\varphi$ is an admissible homeomorphism on $\Sigma$; $U$ is an open set containing $A$; and $\varepsilon >0$. Furthermore, $G$ is a u.s.c. decomposition of $S^n$; $F_1$ and $F_2$ are $F_\sigma$-sets in $\Sigma$ such that $F_1\cup \varphi(F_2)\cup \mathrm{Ext}$ $A$ is $1$-ULC; $A \cap \pi(H_{G}^{*}) = \emptyset$. Then there exists a map $\alpha:S^n \rightarrow S^n$,
\begin{itemize}
\item[(1)]$\alpha|S^n-U=\mathbbm{1}$,
\item[(2)]$\alpha|S^n-A$ is a homeomorphism onto $S^n-\alpha(A)$,
\item[(3)]$\alpha|\Sigma = \alpha|\varphi(\Sigma)$ and $\alpha\varphi$ is a homeomorphism onto $\alpha(A)$ and 
\item[(4)] for $g\in H_G$, $\text{Diam }\alpha(g)<\varepsilon+\text{Diam }g$.
\end{itemize}
\end{lemma}

Before proving Lemma 2, we show how to use it to complete the proof of Lemma 1.

The ($n-1$)-sphere $\Sigma$ in Lemma 2 is $\theta h_C(\text{Bd }C)$, and the homeomorphism $\varphi$ is $h_{\tilde{C}}h_{C}^{-1}\theta^{-1}$. Thus, $\varphi(\Sigma)$ is $h_{\tilde{C}}(\text{Bd }C)=h_{\tilde{C}}(\text{Bd }\tilde{C})$. Given $\eta>0$ such that for any $q,q'\in \tilde{C}$, we have $\rho(h_{\tilde{C}}(q),h_{\tilde{C}}(q'))<\eta$ implies that $\rho(q,q')<\varepsilon$. And $U$ can be defined by $\theta h_{C}(C)\cup A \cup N(h_{\tilde{C}}(\text{Bd }\tilde{C}),\eta)$. 

The boundary of $A$ is the two copies of $\text{Bd }C$ which are disjoint $(n-1)$-spheres. Since $\text{Cl }(S^n-\theta h_C(C))$ is an $n$-cell, its boundary is collared in the cell. Take a middle level of such a collar as $\Sigma_1$. Clearly, $\Sigma_1$ and the boundary of $\text{Cl }(S^n-\theta h_C(C))$ cobound an annulus. Similarly, we can find a $\Sigma_2$ such that $\Sigma_2$ and boundary of $n$-cell $\text{Cl }(S^n-h_{\tilde{C}}(\tilde{C}))$ containing $\Sigma_2$ cobound an annulus. Then $\Sigma_1$ and $\Sigma_2$ cobound an annulus by Annulus Theorem. This means that $A$ is the union of three annuli that fit together nicely.

Next, we shall show $h_{\tilde{C}}h_{C}^{-1}\theta^{-1}$ is admissible. By Bing's Theorem 2 [1], if a point $p$ in the interior of the crumpled $n$-cube $\theta h_C(C)$ is removed, then we can produce retractions of $\theta h_C(C)-\{p\}$ to $\theta h_C(\text{Bd C})$ and of $S^n-\{p\}$ to $\text{Bd }C$. Similarly, we can remove a point $q$ from $h_{\tilde{C}}(\tilde{C})$, and construct retractions of $h_{\tilde{C}}(\tilde{C})-\{q\}$ to $h_{\tilde{C}}(\text{Bd }C)$ and of $S^n-\{q\}$ to $\text{Bd }C$. Hence, in $S^n-{p}-{q}$ there is a short homotopy of $\theta h_C(\text{Bd } C)$ to $h_{\tilde{C}}(\text{Bd }C)$. Since there exists a retraction of $S^n-\{p\}-\{q\}$ goes to $A$, we can find a desired homotopy in $A$ between $\theta h_C|\text{Bd } C$ and $h_{\tilde{C}}|\text{Bd }C$.

By the Mismatch Property, there exist $F_\sigma$ sets $F_1'$ and $F_2'$ in $(n-1)$-sphere $\text{Bd }C$ such that $F_1'\cup \text{Int } C$ and $F_2'\cup \text{Int } \tilde{C}$ are 1-ULC. Hence, $\theta h_C(F_1')$ and $h_{\tilde{C}}(F_2')$ are the desired $F_\sigma$-sets. It follows the hypotheses of Lemma 2 have been satisfied. The conclusion states that a certain map $\alpha$ exists. Now the map
\begin{equation}
h(x)=\begin{cases}
\alpha h_{\tilde{C}} & \text{for } x\in \tilde{C},\\
\alpha\theta h_C & \text{for } x\in C,
\end{cases}
\end{equation}
and $h$ agrees on $C \cap \tilde{C}=\text{Bd }C=\text{Bd } \tilde{C}$ is the required homeomorphism in the conclusion of Lemma 1.

Woodruff's proof of her Lemma 2 relies largely on Eaton's method. Those methods in part depend upon a uniquely $3$-dimensional technique. Daverman [2] combined Eaton's strategy with high dimensional techniques to provide a fundamental shrinking process when the Mismatch Property applies. We will use his work in place of Eaton's.  

For greater clarity, we shall introduce a homeomorphism $\omega: \Sigma \times I \rightarrow A$ such that $\omega(z,0) =z$. Let $\gamma$ denote the homeomorphism from $\Sigma$ to $\varphi(\Sigma)$ given by $\gamma(z) = \omega(z,1)$. We would like to have $\gamma = \varphi$, and we show how to adjust $\omega$ to achieve this. Applying $\omega^{-1}$, projection $q:\Sigma \times I \rightarrow \Sigma$, and admissibility of $\varphi$, we can obtain a homotopy of $\Sigma$ between $\mathbbm{1}|\Sigma$ and $\gamma^{-1} \varphi$. So there exists an isotopy
$\beta_t$ of $\Sigma$ to itself with $\beta_0 =\mathbbm{1}$ and $\beta_1=\gamma^{-1} \varphi$. We can build a homeomorphism $\xi$ of $\Sigma\times I$ to itself given by $\xi(z,t)=(\beta_t(z),t)$. Then $\omega \xi (z,0) =\omega(\beta_0(z),0) = z$
and $\omega \xi(z,1) = \omega(\beta_1(z),1) =\omega(\gamma^{-1}\varphi(z),1)=\varphi(z)$.

Let $\omega\xi=\omega'$. For each $x\in \text{Bd }C$, there exists $z\in \Sigma$ such that 
\begin{equation}
\theta h_C(x)=\omega'(z,0) \text{ and } h_{\tilde{C}}(x)=\omega'(z,1).
\end{equation}

Note $C\cup \tilde{C}$ is topologically equivalent to the decomposition space obtained from the decomposition $\mathcal{G}$ whose only nondegenerate elements are the arcs $\omega(z\times I)$. Also, we say that a motion of a homeomorphism $f$ of $S^n$ to itself \textit{moves points $\delta$-parallel to the fibers}, if to each $s\in S^n$ for which $f(s)\neq s$ there corresponds $z\in S^{n-1}$ such that $\{s, f(s)\}\subset N(\omega(z\times I),\delta)$.

The idea behind our proof is based on Daverman's Controlled Shrinking Lemma [4] or Lemma 4.2 [6] which shows how to shrink one type of arcs in annulus $\omega(S^{n-1}\times I)$. For the convenience of readers, we spell out Lemma 4.2 below,
\begin{lemma}[Daverman's Lemma 4.2] 
Suppose $\omega$ is an embedding of $S^{n-1} \times I$ in $S^n$; $C_e$ is the closure of that component of $S^n-\omega(S^{n-1}\times e)$ not containing $\omega(S^{n-1}\times 1/2)$ $(e=0,1)$; $X$ is a compact subset of $S^{n-1}$ such that $\mathrm{Int}$ $C_0$ is 1-ULC in $C_0-\omega(X\times 0)$; $a\in (0,1]$; $U$ is an open subset of $S^{n-1}$ containing $\omega(X\times [0,a))$. Then for each $\delta>0$ there exists a homeomorphism $f$ of $S^n$ onto itself such that
\begin{itemize}
\item[(1)] $f|(S^n-U)\cup C_1 \cup \omega(S^{n-1}\times [a,1])=\mathbbm{1}$,
\item[(2)] $f\omega(x \times [0,a]) \subset N\omega((x,a),\delta)$, for each $x\in X$,
\item[(3)] $f$ moves points $\delta$-parallel to the fibers of $\omega(S^{n-1}\times I)$. 
\end{itemize}
In addition, if $P=\{t_0,t_1,\dots,t_k\}$ is a partition of $I$, with $0=t_0<t_1<\cdots<t_i<\cdots<t_k=1$, and if $\varepsilon$ is a positive number such that 
$\mathrm{Diam}$ $\omega(z\times [t_{i-1},t_i])<\varepsilon$ for each $z\in S^{n-1}$ and $i=1,\dots,k$, then there exists a homeomorphism $f$ such that
\begin{itemize}
\item[(4)] $\mathrm{Diam}$ $f\omega(z\times [t_{i-1},t_i])<\varepsilon$ for each $z\in S^{n-1}$ and $i=1,\dots,k$.
\end{itemize}
\end{lemma}

\begin{lemma} 
Suppose  $\omega:S^{n-1} \times I  \to S^n$, $C_e$ $(e = 0,1)$, $P=\{ t_0,t_1,\dots ,t_k \}$
and $\varepsilon > 0$ are as in Lemma 3. Suppose that $\mathrm{Diam}$ $\omega(z \times R) < \varepsilon$  for all $z \in S^{n-1}$ and subarcs  $R \subset [t_1,t_{k-2}]$ of diameter  $t_1$.
Suppose also (Mismatch Property) that there exist disjoint  $F_{\sigma}$-subsets
$F_0,F_1$ of $S^{n-1}$ such that  $\mathrm{Int}$ $C_e \cup \omega(F_e \times e)$ is 1-ULC $(e = 0,1)$.
Then for each open subset $V$ of $S^n$  containing  $\omega(S^{n-1} \times I)$ and each $\delta > 0$ 
there exist homeomorphisms  $f:S^n \to S^n$  and $\xi:S^{n-1} \times I \to S^{n-1} \times I$ such that 
\begin{itemize}
\item[(1)] $f | ( S^n - V) \cup \omega(S^{n-1} \times [t_1,t_{k-1}]) = \mathbbm{1}$,
\item[(2)] $f$ moves points $\delta$-parallel to the fibers of $\omega(S^{n-1}\times I)$,
\item[(3)] $\xi$ changes only the second coordinates of points in $S^{n-1} \times I$.
\end{itemize}
In addition, for each  $z \in S^{n-1}$, 
\begin{itemize}
\item[(4)] $\mathrm{Diam}$ $f \omega \xi(z \times [t_{k-2},t_k]) < \varepsilon$, 
\item[(5)] $\mathrm{Diam}$ $f \omega \xi(z \times [0,t_1]) < \varepsilon$ and
\item[(6)] for any subarc  $R \subset [t_1,t_{k-2}]$ of diameter $t_1$, $\mathrm{Diam}$ $f \omega  \xi(z \times R) <\varepsilon$.
\end{itemize}
\end{lemma}

\begin{proof}[Proof of Lemma 4.] The shrinking procedure will be executed on both ends of the annulus, that is, some arcs get shrunk very close to $\omega(S^{n-1}\times [0,t_{k-1}])$; the others get shrunk close to $\omega(S^{n-1}\times [t_1,1])$. Lemma 3 will treat the shrinking of the second type of arcs; Lemma 4 will handle the first type of arcs. 

Fix $V$  and $\delta$. Identify a compact subset  $X$  of $F_1$ such that the 2-skeleton of a very small neighborhood of $\text{Bd }C_1$ admits a highly controlled homotopy, and the image of the end of this homotopy being in $\text{Int } C_1 \cup \omega(X \times 1)$. More will be said about $X$  when we return to it.

In light of the Mismatch features of  $\omega(S^{n-1} \times I)$, Lemma 3 can be applied to obtain a 
homeomorphism $f'$ of  $S^n$ to itself, supported in an open subset $V'$ of $V$ containing $\omega(X \times [t_0,t_1))$, fixing all points of $\omega(S^{n-1} \times [t_1,1])$, moving points $\delta$-parallel to fibers, and sending each arc $\omega(x \times [0,t_1])$  very close to $\omega(x \times t_1)$ ($x \in X$). As a result, one can find a neighborhood  $N$ of $X$ in $S^{n-1}$ such that $f' \omega(z \times [0,t_1])$  is very close to $V\cap \omega(z \times t_1)$  for all $z \in N$.

Identify a neighborhood  $N^*$  of  $X$  with Cl $N^* \subset N$.  Properties of $X$ assure the existence of a homeomorphism  $\phi^*:S^n \to S^n$  that is supported in 
$$V - (V' \cup \omega(S^{n-1} \times [0,t_{k-1}])),$$ 
moves points $\delta$-parallel to fibers, sends each arc $\omega(z \times [t_{k-1},1])$ ($z \in S^{n-1})$ very close to itself and, most importantly, sends each arc $\omega(s \times [t_{k-1},1])$ very close to $\omega(s \times t_{k-1})$ for those  $s \in S^{n-1} - N^*$.

Then $f = \phi^* f'$ will serve as  the desired homeomorphism of $S^n$ to itself. The desired homeomorphism  $\xi$  of $S^{n-1} \times I$ can be defined as the identity on $S^{n-1} \times I$; for points $x \in X$, $\xi$ should send $x \times t_{i}$ to  $x \times t_{i+1}$ ($i = 1,\dots,k-2$) and should be linear on the subintervals of  $x \times I$ bounded by the partition levels; finally, on $N^* - X$, $\xi$ can be determined using an interpolating Urysohn function that sends $z \times [t_1,t_{k-2}] $ isometrically into $z \times [t_1,t_{k-1}]$. 
\end{proof}

\begin{proof}[Proof of Lemma 2.] The net effect takes approximately $k$ repetitions of this procedure to shrink all arcs to small size. The outline of the argument is: Given an embedding $\omega$ as above and $\varepsilon>0$, there exists a positive integer $k$ such that one can shrink the diameters of the arcs to $\varepsilon$-small size as the composition of $k$ $\varepsilon$-homeomorphisms $f_1,\dots, f_k$. Moreover, after $f_1,\dots, f_{i-1}$ have been
determined, one can obtain the next $f_{i}$ as a homeomorphism fixed outside any new open set containing the embedded annulus. Thus, only small elements need be moved by any of these homeomorphisms, and if any element $g\in \mathcal{G}$ becomes dangerously large under $f_{i-1}f_{i-2}\cdots f_1$, it need not be moved by $f_i$.

Choose a partition $P_k=\{0=t_0,t_1,\cdots,t_k=1\}$ of $I$ such that 
\begin{equation}
\text{Diam }\omega(z\times [t_{i-1},t_i])<\varepsilon\text{   for each }z\in S^{n-1}, i=1,\dots,k.
\end{equation}
Arguing inductively on $k$. Select a small neighborhood $V_1$ of $\omega(S^{n-1}\times I)$ containing no nondegenerate element with diameter greater than $\varepsilon/3$. Then find an
$\varepsilon/2$-homeomorphism $f_1$ as required in Lemma 4 of $S^n$ fixed off $V_1$ that moves points $\varepsilon/2$-parallel to the fibers of $\omega(S^{n-1} \times I)$. Note that the diameter of $\omega'(z\times [t_{k-2},t_{k-1}])$ ($\omega'=\omega\xi$) will be shrunk to $\varepsilon/2$-small size after using the homeomorphism $f_1$. As a result, we can eliminate $t_{k-1}$ from $P_k$, form a new partition $P'_{k-1}=\{0=t_0,t_1,\dots,t_{k-2},t_{k-1}=1\}$ of $I$ such that 
\begin{equation}
\text{Diam }f_1\omega'(z\times [t_{i-1},t_i])<\varepsilon/2 \text{   for each }z\in S^{n-1}, i=1,\dots,k-1.
\end{equation}
Pick $V_2$ as a small neighborhood of $f_1\omega'(S^{n-1}\times I)$ so close to the new image of the annulus that $V_2$ hits no big elements of the decomposition $f_1(\mathcal{G})$. Applying Lemma 4 to further shrink $f_1(\mathcal{G})$ by using $f_2$ fixed off $V_2$ such that 
\begin{equation}
\text{Diam }f_2f_1\omega'(z\times [t_{k-3},t_{k-2}])<\varepsilon/2.
\end{equation}
Then a new partition $P'_{k-2}=\{0=t_0,t_1,\dots,t_{k-3},t_{k-2}=1\}$ of $I$ is produced. We have
\begin{equation}
\text{Diam }f_2f_1\omega'(z\times [t_{i-1},t_i])<\varepsilon/2 \text{   for each }z\in S^{n-1}, i=1,\dots,k-2.
\end{equation}
Continuing $k-2$ more applications as above, $V_3,f_3,V_4,f_4,\cdots,f_{k-1}$ are defined inductively. The composition $f=f_{k-1}\cdots f_1$ shrinks the elements of $\mathcal{G}$ to $\varepsilon$-small size while having the feature that
\begin{equation}
\text{Diam }f(g)<\varepsilon+\text{Diam }g, \text{ for all } g\in \mathcal{G}. 
\end{equation}
Lemma 4.1 of [6] assures that if $f_i$ moves points close enough to the fibers of $f_{i-1}f_{i-2}\cdots f_1\omega'(S^{n-1}\times I)$, then $f = f_{k-1}f_{k-2}\cdots f_1$ moves points $\varepsilon$-close to the fibers of $\omega'(S^{n-1}\times I)$, and we choose the $f_i$ with that conclusion in mind.

As a consequence, we get that $\mathcal{G}$ is a shrinkable. So is $\xi(\mathcal{G})$. As a result, there is a map $\alpha:S^n \rightarrow S^n$ $\varepsilon$-close to the identity that realizes the decomposition $\xi(\mathcal{G})$; that is,
$\xi(\mathcal{G}) = \{\alpha^{-1}(s) : s \in S^n \}$. For conclusion (4), the sizes of $\alpha(G)$ are small, as are those of $f(\mathcal{G})$, so a controlled shrink of the latter can be arranged that don't allow sizes of the former to grow very much.

\end{proof}


\begin{thebibliography}{12}
\bibitem{R. H. Bing 1}
    R. H. Bing,
    \textit{Retractions on spheres}, 
     The Amer. Math. Mon. 81 (1964), 482-484.

\bibitem{Cannon, Daverman 2}
    J. W. Cannon and R. J. Daverman, 
    \textit{Cell-like decompositions arising from mismatched sewings: applications to $4$-manifolds}, 
    Fund. Math. 111 (1981), 211-233.

\bibitem{Daverman 3}
    R. J. Daverman, 
    \textit{Decomposition of Manifolds}, 
    Academic Press, Inc. 1986.

\bibitem{Daverman 4}
      \bysame,
      \textit{Every crumpled n-cube is a closed n-cell-complement},
      Michigan Math. J. 24 (1977), No. 2, 225-241.

\bibitem{Daverman 5}
   \bysame, 
   \textit{On the absence of tame disks in certain wild cells},
   Geometric Topology, Lecture Notes in Mathematics Vol. 438 (1975), 142-155.

\bibitem{6}
    \bysame,
    \textit{Sewings of closed n-cell-complements}, 
    unpublished manuscript.

\bibitem{Daverman 7}
            \bysame and D. Repov\v{s},
            \textit{General position properties that characterize $3$-manifolds},
            Canadian J. Math. 44 (1992), 234-251.

        \bibitem{Eaton 8}
         W. T. Eaton,
         \textit{The sum of solid spheres},
         Michigan Math. J. 19 (1972), 193-207.
         
\bibitem{Edwards 9}
    R. D. Edwards,
    \textit{The topology of manifolds and cell-like maps}, 
     In "Proc. Internat. Congr. Mathematicians, Helsinki, 1978" (O. Lehto, ed.), 111-127. Acad. Sci. Fenn., Helsinki, 1980.
         
         
\bibitem{Hosay 10}
   N. Hosay,
   \textit{The sum of a cube and a crumpled cube is $S^3$}, 
   Notices Amer. Math. Soc., 11 (1964), Errata for Vol. 10, 152. 


\bibitem{Lininger 11}
   L. L. Lininger,
   \textit{Some results on crumpled cubes.},
   Trans. Amer. Math. Soc. 118 (1965), 534-549.

\bibitem{Woodruff 12}
         E. P. Woodruff,
         \textit{Decomposition spaces having arbitrarily small neighborhoods with $2$-sphere boundaries},
         Trans. Amer. Math. Soc. 232 (1977), 195-204.
     


\end{thebibliography}
\end{document}